\documentclass{amsart}
\usepackage{verbatim,amssymb,amsmath,amscd,latexsym,amsbsy,mathrsfs} 
\usepackage{graphicx}
\usepackage[all]{xy}
 \input{xy}
 \xyoption{all}

\newtheorem{thm}{Theorem}[section]

\newtheorem{prop}[thm]{Proposition}

\newtheorem{lemma}[thm]{Lemma}

\newtheorem{prob}[thm]{Problem}

\DeclareMathOperator*{\im}{im} 
\DeclareMathOperator*{\coker}{coker} 
\DeclareMathOperator*{\rank}{rank} 
\DeclareMathOperator*{\Spec}{Spec}
\DeclareMathOperator*{\bSpec}{\mathbf{Spec}}
\DeclareMathOperator*{\Ind}{Ind }
\DeclareMathOperator*{\Res}{Res }

\newcommand{\PP}{\mathbb{P}}

\newcommand {\C} {{\mathbb C}}

\newcommand {\Z} {{\mathbb Z}}
\newcommand {\Q} {{\mathbb Q}}

\newcommand {\OO} {{\mathcal O}}

\begin{document}
\title{Simpson's construction of varieties with  many local systems}
\author{
        Donu Arapura    
}
 \thanks {Partially supported by the NSF }
\address{Department of Mathematics\\
 Purdue University\\
 West Lafayette, IN 47907\\
U.S.A.}
 \maketitle

 \begin{center}
   {\em To Steve Zucker}
 \end{center}
One of the goals of this note  is to say something about the fundamental group of a
smooth complex projective variety in terms of the quantity of local systems on
it.  Given a finitely generated group $\Gamma$, let $d_N(\Gamma)$ be the
dimension of the space of irreducible rank $N$ representations.  The number 
$d_1(\Gamma)$ coincides with the first Betti number,
so one may think of $d_N(\Gamma)$ as a
nonabelian generalization. The basic problem is to see how these numbers behave when $\Gamma$ is the
fundamental group of a smooth projective variety $X$. In this case, these numbers
are always even \cite{arapura}.
If $X$ is a curve of genus at least two, or even if it maps onto such a
curve, then $d_N(\Gamma)>0$ for all $N$. If $X$ is an abelian variety,
then $d_1(\Gamma)>0$  but $d_N(\Gamma)=0$ for all $N>0$. I want to consider examples which
 have the opposite behaviour, in that 
 $d_1=0$ but some higher $d_N>0$. Some cheap examples are given in
the first section. However, they are not very interesting in the sense that
they are very close to the examples we already know.
In the second section I will turn to a beautiful construction due
to  Carlos Simpson \cite{simpson}, which also  produces smooth
projective varieties such that  $d_N(\pi_1(X))>0$ for some $N>1$.
In fact, the real purpose of  this article  is to make Simpson's construction a bit more
accessible and explicit, with the hope that these examples will be studied more
 thoroughly in the future. Some specific problems are suggested in the
 last section.

\section{Representation varieties}

For $\Gamma$ a group with generators $g_1,\ldots, g_n$, 
an element of $Hom(\Gamma, GL_N(\C))$ is given by $n$ matrices subject
to the relations of the group. In this way, the set becomes an affine
scheme of finite type, called the representation ``variety''. (For the present purposes,
a scheme will be identified with the set of its closed points.) The
algebraic group $GL_N(\C)$ acts on the representation variety by conjugation,
and the GIT quotient
\begin{equation*}
  \begin{split}
M(\Gamma, N) &= Hom(\Gamma, GL_N(\C))//GL_N(\C)    \\
&:= \Spec \OO( Hom(\Gamma, GL_N(\C)))^{GL_N(\C)}  
  \end{split}
\end{equation*}
can be identified with the set of isomorphism classes of {\em semisimple} representations of
rank $N$ \cite{lm}. This is often called the character variety. Let
$$M(\Gamma, N)^{irred}\subset M(\Gamma, N)$$
denote the possibly empty open subset of irreducible representations.
We have quasifinite (i.e. set theoretically finite to one) morphisms
$$M(\Gamma, N_1)\times M(\Gamma, N_2)\to M(\Gamma, N_1+N_2)$$
given by direct sum. We can decompose
\begin{equation}
  \label{eq:decomp}
M(\Gamma, N) = \bigcup_{N_1+\ldots +N_r=N} \im M(\Gamma,
N_1)^{irred}\times\ldots \times M(\Gamma,N_r)^{irred}
\end{equation}
Let 
$$d_N(\Gamma) = \dim M(\Gamma, N)^{irred}$$
where we take it to be zero if it is empty. From \eqref{eq:decomp}, we
obtain:

\begin{lemma}\label{lemma:decomp}
We have
$$\dim M(\Gamma, N) = \max_{N_1+\ldots +N_r=N}  d_{N_1}(\Gamma)+\ldots
+ d_{N_r}(\Gamma)$$
  Therefore $\dim M(\Gamma, N)>0$ if and only if $d_M(\Gamma)>0$ for some $M\le N$.
\end{lemma}

We have $M(\Gamma, 1)=\dim Hom(\Gamma,\C^*)$, therefore $d_1(\Gamma) =
\rank \Gamma/[\Gamma, \Gamma]$.
 For higher $N$,
these numbers are usually very difficult to calculate, although there
are some easy cases. We have 
$d_N(\Gamma)=0$ when $N>1$ and $\Gamma$ is abelian simply because in this case
 there are no irreducible representations of higher rank.  
If $\Gamma$ surjects onto a nonabelian free group then a bit of thought
shows that $d_N(\Gamma) >0$ for all $N$.
This remark applies to   the fundamental group of a smooth
projective curve of genus at least two. 

When $\Gamma=\pi_1(X)$ is the fundamental group of a smooth projective
variety $X$,  Hodge theory tells us that $d_1(\Gamma)=\dim H^1(X)$ is
even.  More generally, nonabelian Hodge theory implies
that $M(\Gamma, N)^{irred}$ carries a quaternionic or
hyperk\"ahler structure, therefore every $d_N(\Gamma)$ is even \cite[thm 3.1]{arapura}.
Here is the example promised in the introduction.

\begin{thm}
  There exists a smooth projective variety $X$ with $d_1(\pi_1(X))=0$
  and $d_N(\pi_1(X))\ge 2d$ for any given $N>1$ and $ d>0$.
\end{thm}

\begin{proof}
Let $C\to \PP^1$ be a cyclic cover of the form $y^N=
f(x)$, where $f$ has distinct roots. Let $x_0$ denote one of the roots.
 By choosing $\deg f$ sufficiently
large, we can assume that the genus $g$ of $C$ is greater than or
equal to $d$.
The group  $G=\Z/N\Z$ will act on $C$ with   $C/G\cong \PP^1$. 
If follows that $H^1(C,\Q)^G=0$. Consequently, if $\gamma\in G$ denotes a generator,
it will act nontrivially on $H_1(C,\Z)$.
By Serre \cite[prop 15]{serre}, there exists a simply connected variety $Y$ on which
$G$ acts freely. Let $X = (C\times Y)/G$, where $G$ acts
diagonally.  The projection $X\to Y/G$ is a fibration with fibre $C$
and section  given by $y\mapsto (x_0,y)$.
Therefore we have split exact sequence
$$1\to \pi_1(C)\to \pi_1(X) \stackrel{\leftarrow}{\to} G\to 1$$
Using the Hochschild-Serre spectral sequence, we obtain an exact sequence
$$ H^1(G, H^0(\pi_1(C),\Q))\to H^1(\pi_1(X),\Q){\to} H^0(G,
H^1(\pi_1(C),\Q))$$
As noted above, the group on the
right  vanishes.
Since $G$ is finite, the group on the left also vanishes. Therefore $H^1(\pi_1(X), \Q)=0$, which means that
$d_1(\pi_1(X))=0$.

Let $\rho\in Hom(\pi_1(C),\C^*)= (\C^*)^{2g}$ be a one dimensional
character. For a  generic choice of $\rho$, the characters $\rho,
\rho\circ\gamma, \ldots \rho\circ \gamma^{N-1}$ are all distinct.
Let $\C_\rho$ denote the $\C[\pi_1(C)]$-module associated to $\rho$. The
induced representation $V_\rho=\Ind \C_\rho$ gives a rank $N$
$\C[\pi_1(X)]$-module. As an $\C[\pi_1(C)]$-module 
\begin{equation}
  \label{eq:V}
V_{\rho}=\C_\rho\oplus\C_{\rho\circ \gamma}\oplus\ldots  
\end{equation}
and $\gamma$ acts by cyclically permuting the factors.
It follows easily that $V_\rho$ is an irreducible $\pi_1(X)$-module for generic $\rho$.  Also by computing 
characters, using \eqref{eq:V}, we see that $V_\rho\cong V_{\rho'}$  only if $\rho'\in \{\rho, \gamma\rho,\ldots\}$.
Therefore the  map $\rho\mapsto
V_\rho$ is a quasifinite morphism from an open subset of $(\C^*)^{2g}$ to
$M(\pi_1(X), N)^{irred}$. 
Thus  $d_N(\pi_1(X))\ge 2g\ge 2d$.
\end{proof}

The drawback of this method is that it does not produce any
really new examples of fundamental groups of smooth projective
varieties. I will describe a more subtle construct in the next
section, but first I want to  record the following useful fact which was used
implicitly above.

\begin{lemma}\label{lemma:induced}
Suppose that $\Gamma_1\subset \Gamma$ is a subgroup of  index $r<\infty$.
\begin{enumerate}
\item[(a)] If  $W_\rho$ is  a nontrivial (i.e. nonconstant) family of representations in $M(\Gamma, N)$, then
  the restrictions $\Res W_\rho$ give a nontrivial family in
  $M(\Gamma_1,N)$.
\item[(b)] Conversely if  $\Res W_\rho$ is a nontrivial family, then so is $W_\rho$.
\item[(c)] If  $V_\rho$ is  a nontrivial    family of representations in $M(\Gamma_1, N)$, then
  $\Ind V_\rho$ is a nontrivial family in $M(\Gamma,rN)$
\end{enumerate}
\end{lemma}

\begin{proof}
The first two items are the content of lemma 1.5 of
\cite{simpson}. For (c), we have that $\Res(\Ind V_\rho) =
V_\rho\oplus\ldots$ is nontrivial, so $\Ind V_\rho$ is nontrivial by (b).
\end{proof}

\section{Simpson's construction}

Let $Z$ be a smooth projective variety with  dimension $2n+1\ge 3$ and positive first
Betti number.  Fix an embedding $Z\subset\PP^K$ such that $\OO_Z(1)$ is sufficiently ample
in the sense that it is a high enough power of a given ample
bundle. Sufficient ampleness is needed for the proofs of  proposition \ref{prop:disc} and
theorem \ref{thm:simpson}.
 Let $P\subset \check{\PP}^K$ be a general linear subspace of the dual
 space of dimension $d\ge
2$. Then we can form the incidence variety
$$Y=\{(z,H)\in Z\times P\mid z\in H\}$$
with projections and inclusions labelled as follows
$$
\xymatrix{
 Y\ar[r]^{f}\ar[d]^{\pi}\ar[rd]^{\iota} & Z \\ 
 P & Z\times P\ar[u]_{F}\ar[l]^{\Pi}
}
$$
Denote the fibre of $\pi$ over $t$ by $X_t$. Let $D_1=\{t\in P\mid X_t\text{
  is singular}\}$ be the (reduced) discriminant.  The following
standard fact  is stated in \cite{dl} and various other places.
A proof, assuming sufficient ampleness, can be found in \cite[prop 6.1]{simpson}. 

\begin{prop}\label{prop:disc}
 The discriminant  $D_1$ is an irreducible hypersurface and
  for a generic $2$ dimensional plane $Q\subset P$, the singularities of
  $D_1\cap Q $ are nodes and cusps.
\end{prop}

The next step is to form a double cover branched over $D_1$.
If $g(x_1,\ldots, x_d)=0$ is an affine equation of $D_1$, then the
 cover $y^2=g$ may acquire additional ramification at infinity.
It is better to control this in advance by defining
$$
D =
\begin{cases}
  D_1 & \text{if $\deg D_1$ is even}\\
 D_1+D_2 &\text{otherwise, where $D_2$ is a hyperplane in general position}
\end{cases}
$$
Let $U= P-D$. 
Let $p':X'\to P$ be the double cover branched along $D$.
As a scheme
$$X' = \bSpec\left(\OO_P\oplus \OO_P\left(\frac{\deg D}{2}\right)\right)$$
where the sheaf in parantheses is made into an algebra in the
standard way (cf \cite[p 22]{ev}). This will usually be singular but
the singularities are 
normal local complete intersections.  The singular set
$X'_{sing}\subseteq \Sigma={p'}^{-1}D_{sing}$.
Let $q:X\to X'$ be a desingularization  which is an isomorphism on the
complement of $\Sigma$.  This variety is what we are after. It is very
similar to, although not identical to,   Simpson's construction in \cite[lemma
6.3]{simpson}. The difference is that Simpson's variety is a
branched cover of $P$ of indeterminate degree, on which,  by design,
the local systems $V_\rho$ constructed below extend. This makes 
it simpler  for the purpose of constructing local systems. However, the lack of
explicitness makes it harder to do precise computations.

\begin{thm}\label{thm:main}
  The first Betti number of $X$ is zero. For some $M>1$, $d_M(\pi_1(X))>0$.
\end{thm}

The rest of this section will be devoted to the proof of this theorem.

\begin{prop}\label{prop:betti}
  The first Betti number of $X$ is zero. 
\end{prop}

\begin{proof}
By Hodge theory, the proposition  is equivalent to $H^1(X,\OO_X)=0$.
We prove the last equation by induction on $d$ starting with $d=2$. In this case,  
$\Sigma$ consists of a finite set of
singular points.   The local analytic germ of $X'$ at  $p\in \Sigma$ is either  of the form $y^2=x_1x_2$ or
$y^2=x_1^2-x_2^3$. These are the  well known singularities of type
$A_n$  for $n=1,2$ \cite{durfee}.
These are rational singularities which implies that $H^1(X,\OO_X)=
H^1(X',\OO_{X'})$.   The last group
$$H^1(X',\OO_{X'})\cong H^1(\PP^2, \OO_{\PP^2})\oplus
H^1(\PP^2,\OO_{\PP^2}(\deg D/2))=0$$

For $d>2$, choose a general hyperplane $H\subset P$. By the Bertini,
$G=p^{-1}H$ is smooth. By induction, we can assume that
$H^1(G,\OO_G)=0$.
 We have an exact sequence
$$H^1(X,\OO_X(-G))\to H^1(X,\OO_X)\to H^1(X,\OO_G)=0$$
The first group $H^1(X,\OO_X(-G))= H^{d-1}(X,\omega_X(G))$
is zero  by the Kawamata-Viehweg vanishing theorem \cite[p 49]{ev}.

\end{proof}

We turn to the second part of theorem. 
By assumption $Z$ carries a positive dimensional family of rank one
local systems. Fix a generic such system $\C_\rho$, and 
consider the sheaf 
$$V_\rho= \coker( R^n\Pi_*(F^*\C_\rho)\stackrel{\iota^*}{\to} R^n\pi_*(f^*\C_\rho))|_U$$
 This is a local
system of some rank $N>1$. The stalk of $V_\rho$ over $t$ is the
primitive $n$th cohomology of $X_t$ with coefficients in $\C_\rho$.
The rank $N$ is just the dimension of this space. Let
$R_\rho:\pi_1(U)\to GL_N\C$ denote
the representation corresponding to $V_\rho$.

\begin{thm}[Simpson {\cite[ thm 5.1]{simpson}}]\label{thm:simpson}
 As $\rho$ varies, $V_\rho$ gives a nontrivial family in
  $M(\pi_1(U), N)$.
\end{thm}

The proof is rather involved, so we will be content to make a few brief
comments about it. The key ingredient is
 nonabelian Hodge theory,  which sets up a correspondence between
semisimple local systems and certain Higgs bundles, which for our
purposes can be viewed as sheaves on the cotangent bundle.
Simpson then checks that as the $\rho$ vary, the supports of the Higgs bundles
corresponding to $V_\rho$, called spectral varieties, also vary  nontrivially.  When $Z$ is an abelian variety,
there is a more elementary argument which avoids Higgs bundles \cite[p
358]{simpson}, and this   already suffices for constructing nontrivial examples.

Let $\gamma_1$ be a loop going once around a smooth point $D_1$. 
This involves a choice, but any
two choices are conjugate because $D_1$ is irreducible.  We have
\begin{equation}
  \label{eq:Rgamma1}
R_\rho(\gamma_1)^2=I  
\end{equation}
by the Picard-Lefschetz formula or see \cite[lemma 6.5]{simpson}.
Let $\gamma_2$ be a loop around $D_2$ when it exists.  Then
\begin{equation}
  \label{eq:Rgamma2}
R_\rho(\gamma_2)=I  
\end{equation}
 because $V_\rho$ extends to a local system on
$P-D_1$. Let $p=p'\circ q$ and $\tilde U=p^{-1}U$. We can identify $\tilde
U={p'}^{-1}U\subset X'$. This  is an \'etale double cover of $U$
corresponding to an index two subgroup $\pi_1(\tilde U)\subset
\pi_1(U)$. This subgroup contains $\gamma_i^2$. We
can identify $\pi_1(X'-\Sigma)$ with the quotient of $\pi_1(\tilde U)$
by the normal subgroup generated by the $\gamma_i^2$. Combining this with
\eqref{eq:Rgamma1} and \eqref{eq:Rgamma2} yields

\begin{lemma}
  The pullback of the local system ${p'}^*V_\rho$ extends to $X'-\Sigma$.
\end{lemma}

Let $X'_Q =  X'\cap {p'}^*Q$ where $Q\subset P$ is a general
   $2$-plane.  

   \begin{lemma}
     $\pi_1(X')\cong
   \pi_1(X'_Q)$ 
   \end{lemma}

   \begin{proof}
     Since $X$ has local complete
   intersection singularities, we can apply the Lefschetz theorem of
   \cite[p 28]{fl} to deduce the above isomorphism.
   \end{proof}

To simplify notation,  replace $\Sigma$  by
   its restriction to $X'_Q$.
 Then $\Sigma$ consists of a finite set of points. For
   each $p\in \Sigma$, let $L_p$ denote the link which is the boundary
   of a small contractible neigbourhood of $p$. The
 group $\pi_1(X')=\pi_1(X'_Q)$ is the quotient of $\pi=\pi_1(X'_Q-\Sigma)$ by the normal
 subgroup $N$ generated by $\bigcup_p \pi_1(L_p) $. For any group
 $\Gamma$, let
$$K(\Gamma)= \ker [\Gamma\to \widehat{\Gamma}]$$ 
where $\widehat{\Gamma}$ is the profinite completion.
This can also be characterized as the intersection of all finite index
subgroups, or as the  smallest normal subgroup for which $\Gamma/K(\Gamma)$  is residually
finite.

\begin{lemma}\label{lemma:GammaExists}
  There exists a normal subgroup of finite index $\Gamma\subseteq \pi$ such
  that $\pi_1(L_p)\cap \Gamma\subseteq K(\pi)$ for each $p\in \Sigma$.
\end{lemma}

\begin{proof}
As noted above, $\Sigma$ consists of a finite set of
singular points of type $A_1$ or $A_2$. These
singularities can also be described as quotients of $(\C^2,0)$ by an
action of $\Z/2\Z$ or
$\Z/3\Z$ \cite{durfee}. Therefore $\pi_1(L_p)$ must either be $\Z/2\Z$ or
 $\Z/3\Z$ and in particular finite. Since $\pi/K(\pi)$ is residually
 finite, we can find a finite index subgroup $\bar
\Gamma$  of it avoiding the nonzero elements $\im(\pi_1(L_p))$. Let
$\Gamma$ be the preimage.
\end{proof}

Let $\Psi \subseteq \pi_1(X')$ denote the image of $\Gamma$.

\begin{lemma}
\-
  \begin{enumerate}
\item[(a)] If $\bar \Gamma$ and $\bar N$ denote  the images of $\Gamma$ and $N$ in
  $\pi/K(\pi)$, then  $\bar \Gamma\cap \bar N=1$.
\item[(b)] $\Gamma/K(\pi)\cong \Psi/K(\pi_1(X'))$.
\end{enumerate}
\end{lemma}

\begin{proof}
Item (a) follows immediately  from lemma~\ref{lemma:GammaExists}. 
The canonical map  $\Gamma/K(\pi)\to \Psi/K(\pi_1(X'))$ is clearly
surjective. The kernel is $\bar \Gamma\cap \bar N$. So (b) follows
from (a).  
\end{proof}

\begin{lemma}\label{lemma:malcev}
 The restriction  of $R_\rho$ to $\Gamma$ is the pull back of a representation of
  $\Psi$. 
\end{lemma}

\begin{proof}
  By a theorem of Mal\v{c}ev \cite[p 309]{magnus},  any finitely generated
  linear group is residually finite. Therefore the restriction $\Res V_\rho= R_\rho|_\Gamma$
  factors through $\Gamma/K(\pi)\cong \Psi/K(\pi_1(X'))$.
\end{proof}

To finish the proof of theorem~\ref{thm:main}, observe that by the above
results, the restriction $\Res V_\rho$ comes from a $\Psi$-module $W_\rho$.
We can form
the induced $\pi_1(X')$-module $\Ind W_\rho$. This corresponds to  a nontrivial family
of semisimple local systems on $X'$ by lemma~\ref{lemma:induced}, which pulls back to a
nontrivial family on $X$. Therefore by lemma~\ref{lemma:decomp}
$d_M(\pi_1(X))>0$ for some $M$. By proposition~\ref{prop:betti},
$M>1$, and this concludes the proof.

\section{Problems}

I will end by  discussing  a few follow up  problems.

\begin{prob}
  Determine (a presentation for) the
fundamental group of $X$, constructed in section two, for some explicit choice of $Z\subset \PP^K$,
such as when it is an abelian variety.
\end{prob}
 
My hope is that this  will give a genuinely new and interesting
example of a group in the class of  fundamental
groups of smooth projective varieties.  It is clear that it would
differ from most of the standard known examples which 
either have positive first Betti number or are rigid in the sense that
all $d_N=0$. Furthermore $\pi_1(X)$ would be different from the examples
constructed in section one. Simpson's arguments \cite{simpson} show that in his
terminology  that $X$, with the local system $\mathcal{I}=q^*\Ind W_\rho$
above,  has the nonfactorization
property  $NF_1$. This means that  $\mathcal{I}$ is not  the pull back of a
local system on  a curve even if we  allow $X$ to be replaced by another variety mapping
surjectively to  it.  This will imply that $\pi_1(X)$
cannot contain the fundamental group of a
curve as a subgroup of finite index. 

\begin{prob}
  Find an example of a smooth projective variety with an infinite
  family of irreducible {\em unitary} representations which do not come from
  curves, i.e. that satisfy $NF_1$
\end{prob}

This is equivalent to asking for a variety with an infinite family of
stable vector bundles, with vanishing Chern classes, which do not come
from curves. This can be rephrased as asking for an infinite family of stable Higgs
bundles of the above type with zero Higgs fields.
Simpson's construction described above yields Higgs bundles with
nonzero Higgs fields. This is clear from his proof of theorem \ref{thm:simpson}.

For applications to the fundamental group, it suffices to stick with
dimension $d=2$.
One reason for allowing  $d> 2$ is that  I feel that
these varieties should be interesting from other points of view.

\begin{prob}
  Study the birational geometry of these varieties.
\end{prob}

For instance, although they have zero first Betti number,  I suspect
that they behave like varieties with large Albanese. One way to try
to make this precise is by using the notion of Shafarevich maps in the sense of
Campana and Koll\'ar \cite{kollar}. In most cases, I suspect that this
map should be birational. This would be an analogue of the Albanese
map being generically finite.

\end{document}